\documentclass[letterpaper, 10 pt, conference]{ieeeconf}  % Comment this line out if you need a4paper

\IEEEoverridecommandlockouts                              % This command is only needed if 
\usepackage{graphics} % for pdf, bitmapped graphics files
\usepackage{epsfig} % for postscript graphics files
\usepackage{mathptmx} % assumes new font selection scheme installed
\usepackage{times} % assumes new font selection scheme installed
\usepackage{amsmath} % assumes amsmath package installed
\usepackage{amssymb}  % assumes amsmath package installed
\usepackage{dsfont}
\usepackage{mathtools}
\usepackage[utf8]{inputenc}
\usepackage{graphicx}
\usepackage{newtxtext, newtxmath}

\usepackage{enumerate}
\newcommand{\R}{\mathds{R}}
\newcommand{\A}{\mathds{A}}
\newcommand{\1}{\mathds{1}}
\newcommand{\N}{\mathds{N}}

\let\phi\varphi

\newcommand{\Z}{\mathds{Z}}

\newcommand{\sB}{\mathcal{B}}

\newcommand{\sM}{\mathcal{M}}
\newcommand{\sW}{\mathcal{W}}
\newcommand{\sX}{\mathcal{X}}

\newcommand{\sF}{\mathcal{F}}

\newtheorem{definition}{Definition}
\newtheorem{theorem}{Theorem}
\newtheorem{remark}{Remark}
\newtheorem{example}{Example}
\newtheorem{lemma}{Lemma}

\usepackage[colorlinks=true,breaklinks=true,bookmarks=true,urlcolor=blue,
     citecolor=blue,linkcolor=blue,bookmarksopen=false,draft=false]{hyperref}

\title{\LARGE \bf
Ergodicity of Controlled Stochastic Nonlinear Systems under Information Constraints: Refined Bounds via Splitting
}

\author{Nicolas Garcia$^{1}$, Christoph Kawan$^{2}$, and Serdar Y\"uksel$^{3}$% <-this % stops a space
\thanks{*The work of the second author is supported by the German Research Foundation (DFG)
through grant ZA 873/4-1.}% <-this % stops a space
\thanks{$^{1}$Is with the Department of Operations Research and Financial Engineering at Princeton University, Princeton NJ, USA.
        {\tt\small ng6303@princeton.edu }}%
\thanks{$^{2}$Is with the Institute of Informatics at the LMU Munich, Germany.
        {\tt\small christoph.kawan@lmu.de}}%
\thanks{$^{3}$Is with the Department of Mathematics and Statistics at Queen's University, Kingston ON, Canada.
        {\tt\small yuksel@queensu.ca}}%
}

\begin{document}

\maketitle
\thispagestyle{empty}
\pagestyle{empty}

%%%%%%%%%%%%%%%%%%%%%%%%%%%%%%%%%%%%%%%%%%%%%%%%%%%%%%%%%%%%%%%%%%%%%%%%%%%%%%%%
\begin{abstract}
This paper considers the problem of stabilizing a discrete-time non-linear stochastic system over a finite capacity noiseless channel. Our focus is on systems which decompose into a stable and unstable component, and the stability notion considered is asymptotic ergodicity of the $\R^N$-valued state process. We establish a necessary lower bound on channel capacity for the existence of a coding and control policy which renders the closed-loop system stochastically stable. In the literature, it has been established that under technical assumptions, the channel capacity must not be smaller than the logarithm of the determinant of the system linearization, averaged over the noise and ergodic state measures. In this paper, we establish that for systems with a stable component, it suffices to consider only the unstable dimensions, providing a refinement on the general channel capacity bound for a large class of systems. The result is established using the notion of stabilization entropy, a notion adapted from invariance entropy, used in the study of noise-free systems under information constraints.	

\end{abstract}

%%%%%%%%%%%%%%%%%%%%%%%%%%%%%%%%%%%%%%%%%%%%%%%%%%%%%%%%%%%%%%%%%%%%%%%%%%%%%%%%
\section{INTRODUCTION}

In the field of control under communication constraints, a commonly studied problem is to characterize the minimum amount of information required by a controller in order to achieve a given control task. In this paper we consider the above problem for discrete-time non-linear stochastic systems with additive control. The control objective considered is to render the state process stochastically stable for the stability criterion of asymptotic ergodicity. More precisely, we consider the system
\begin{align} \label{system}
x_{t+1} = f(x_t,w_t)+Bu_t
\end{align}
where $x_t,w_t$ and $u_t$ are the state, noise, and control at time $t$ respectively and $B$ is an appropriately sized matrix. Additionally, we impose that the state information  travel through a finite capacity noiseless channel at each time step before reaching the controller, as depicted in Figure \ref{LLL1}. We formalize the notion of a coding and control policy as follows. First, let $\sM  \coloneqq  \{1,2,..,M\}$ denote the alphabet of the channel, thus its capacity in bits is given by $C \coloneqq \log_2 M$. At time $t$, the coder (also known as the encoder) generates a channel input $q_t$ from past state realizations $x_0,..,x_t$. The channel input $q_t \in \sM$ is therefore determined by a map $\gamma_t^e:(\R^N)^{t+1} \to \{1,2,\ldots,M\}$. The symbol $q_t$ is transmitted over the channel, reaching the controller. The controller generates $u_t$ based on channel outputs $q_0,..,q_t$ according to a map $\gamma_t^c : \sM^{t+1} \to \R^N$. A coding and control policy is therefore a pair of maps $(\gamma_t^e)_{t \in \N}$ and $(\gamma_t^c)_{t \in \N}$. Once we fix a coding and control policy, $(x_t)_{t \in \N}$ is a well defined autonomous stochastic process, with randomness coming from the possibly random initial state $x_0$, and the noise process $(w_t)_{t \in \N}$. 
\begin{figure}[h]
\begin{center}
\includegraphics[height=3.5cm,width=8.5cm]{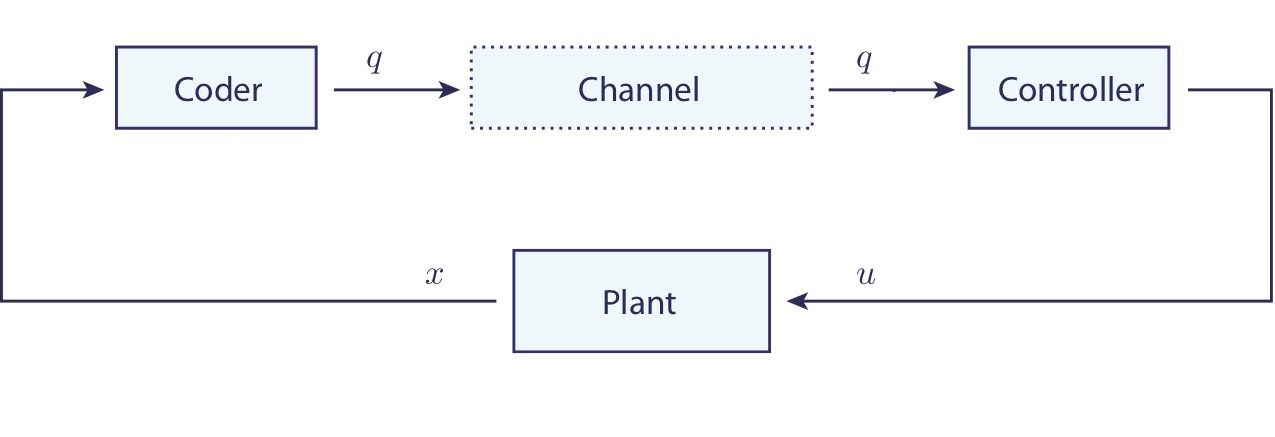}
\caption{System Controlled Over Communication Channel \label{LLL1}}
\end{center}
\end{figure}

In this paper, we establish a necessary lower bound on channel capacity required for the existence of coding and control policies which result in the state process $(x_t)_{t \in \N}$ being asymptotically ergodic. The lower bound is related to the asymptotic mean (defined in the next section), i.i.d. noise law, and system dynamics function. We now discuss the stability notion in detail, but first provide some notational conventions.

\subsection{Notation}
Throughout this paper, $\Z$ denotes the integers, $\R$ the real numbers, and $\N$ the non-negative integers. The Lebesgue measure is denoted by $m$, where the dimension will be clear from context. A discrete interval in the integers will be denoted by $[a;b]$ (i.e., $[a;b] = \{a,a+1,\ldots,b-1,b\}$ for $a \leq b$ in $\Z$). Given a topological space $\sX$, $\sB(\sX)$ denotes its Borel $\sigma$-algebra. For a function $f:\R^n \rightarrow \R^m$, we denote the Jacobian (matrix of partial derivatives) by $Df$. We use $\sqcup$ to emphasize that a union in question is disjoint. When applied to a set, $|\cdot|$ denotes cardinality. Given a sequence $x \coloneqq (x_n)_{n \in \N}$ taking values in a given space, $\theta$ denotes the left shift map, so that $(\theta x)_n = x_{n+1}$ for every $n \in \N$. Given a topological space $\sX$, we let $\sX^{\N}$ denote the set of sequences indexed by $\N$, taking values in $\sX$. We endow  $\sX^{\N}$ with the product topology. 

\subsection{Asymptotic Ergodicity}
In this section we discuss the stochastic stability notion considered in this paper. Let $(\Omega,\sF,P)$ denote the common probability space on which all random variables are defined. Note that fixing a coding and control policy and specifying an initial state distribution for the system (\ref{system}) results in an autonomous state process $(x_t)_{t \in \N}$ which induces a process measure on $\sB((\R^N)^{\N})$. Let us denote this measure by $\mu$. 

\begin{definition} \label{stochstabilitydef}
Consider the process $(x_t)_{t \in \N}$ with process measure $\mu$ as above. We say it is:
\begin{itemize}
\item \emph{stationary} iff $\mu(\theta^{-1}(A)) = \mu(A)$ for all $A \in \sB((\R^N)^{\N})$.%
\item \emph{asymptotically mean stationary (AMS)} iff there exists a probability measure $Q$ (called the asymptotic mean of the process) on $\sB((\R^N)^{\N})$ such that%
\begin{equation*}
    \lim\limits_{T\rightarrow \infty} \frac{1}{T}\sum_{k=0}^{T-1} \mu(\theta^{-k}(A)) = Q(A) \quad \mbox{for every\ } A \in \sB((\R^N)^{\N}).%
\end{equation*}
\item \emph{ergodic} iff it is stationary, and for $A \in \sB((\R^N)^{\N})$ we have that $A = \theta^{-1}(A) \implies \mu(A) \in 
\{0,1\}$.
\item \emph{AMS ergodic} iff it is AMS, and the asymptotic mean is ergodic.
\end{itemize}
\end{definition}
Note that if a process is \emph{AMS}, then the asymptotic mean is a stationary measure on the sequence space. Note also that a stationary measure on $\sB((\R^N)^{\N})$ can unambiguously be projected to a measure on $\sB(\R^N)$. By slight abuse of notation, we do not distinguish between a stationary measure on the sequence space and its projected coordinate measure, as the measure in question will be clear from context. In this paper, the stability notion that we will consider is \emph{AMS ergodicity} (or informally, \emph{asymptotic ergodicity}). Ergodicity allows us to take advantage of the following pointwise ergodic theorem:

\begin{theorem}
(Pointwise Ergodic Theorem) Let $(\Omega,\sF,P,T)$ be an ergodic system. Then for any $f \in L^1(\Omega,\sF,P)$ we have%
\begin{equation*}
  \frac{1}{N}\sum_{k=0}^{N-1} f \circ T^k \xrightarrow[N \rightarrow \infty]{a.s} \int f d P.%
\end{equation*}
\end{theorem}
For a proof, see \cite[Theorem 1.14]{walters2000introduction}.

Suppose that the stochastic process in Definition \ref{stochstabilitydef} is \emph{AMS ergodic} with process measure $\mu$ and asymptotic mean $Q$. The application of the pointwise ergodic theorem to the $L^1$ map $\1_{x_0 \in B} : (\R^N)^{\N} \to \{0,1\}$ for $B \in \sB(\R^N)$ yields 
\begin{align}\label{spce}
  \mu\Bigl(\Bigl\{x \in (\R^N)^{\N} : \lim_{T \to \infty}\frac{1}{T}\sum_{k=0}^{T-1}\1_{B}(x_k) = Q(B) \Bigr\}\Bigr) = 1,%
\end{align}
In principle, the pointwise ergodic theorem tells us that the above set has full measure with respect to $Q$. From \cite[Lem. 7.5 and Eq. (7.22)]{gray2009probability} it turns out that $\mu$ and $Q$ agree on all $Q$-trivial sets, thus allowing one to write (\ref{spce}). The above equation which provides almost sure guarantees on asymptotic sample path behavior will be crucial in the results of this paper, for which we now provide some motivating examples.  

\subsection{Problem Motivation}
Suppose that system (\ref{system}) is controlled over a noiseless channel of finite capacity $C$, and made to be asymptotically ergodic with AMS mean $Q$. Then under slightly different technical assumptions, \cite{garcia2021ergodicity} and \cite{yuksel2016stationary} establish that
\begin{align} \label{oldbound}
\int \int \log_2|\det Df_w(x)|dQ(x)d\nu(w) \leq C
\end{align}
where $\nu$ is the distribution of the i.i.d. noise and $f_w$ denotes the map $x \mapsto f(x,w)$ for some fixed $w \in \sW$. This bound however is in general not tight, as the following two examples illustrate. 
\begin{example} \label{ex1bd}
Consider the two dimensional linear system given by
\begin{align} 
 \begin{bmatrix} x_{t+1} \\ y_{t+1} \end{bmatrix} = 
 \begin{bmatrix} 2 & 0 \\ 0 & 1/2 \end{bmatrix} 
 \begin{bmatrix} x_{t} \\ y_{t} \end{bmatrix}+ w_t + u_t
\end{align}
where $u_t$ and $w_t$ take values in $\R^2$ and the noise is i.i.d with zero mean. The LHS of (\ref{oldbound}) is easily seen to equal zero for this system, thus providing a vacuous bound on channel capacity. It is well known from the literature however that a tight bound on linear systems is the log-sum of the unstable eigenvalues. Note that by replacing $1/2$ in the above matrix with any number no smaller than one, the bound in (\ref{oldbound}) recovers the tight linear bound. As we will see, the refinement of the channel capacity bound in this paper will recover the tight bound in the general linear case (thus, also with stable eigenvalues). Consider now a second example:  

\end{example}

\begin{example} \label{ex2bd}
Consider the system $(x_t,y_t)_{t \in \N}$ in $\R^2$ evolving with scalar-valued i.i.d. noise according to
\begin{align}
\begin{bmatrix} x_{t+1} \\ y_{t+1} \end{bmatrix} = \begin{bmatrix} (x_t^3 + x_t)(1+y_t^2) \\ \frac{1}{2}y_t + w_t \end{bmatrix}+ u_t
\end{align}
with $x_0$ and $y_0$ independent and admitting bounded densities. We note that the $y$-component of the above system is stochastically stable. Moreover, the presence of the $y$-term in the dynamics of the $x$-component cannot be modeled as noise, as the i.i.d. assumption required in data-rate theorems of non-linear systems is not satisfied. Suppose the above system is made asymptotically ergodic via a coding and control policy with AMS mean $Q$. We compute
\begin{align}
Df_w(x,y) = 
\begin{bmatrix}
(3x^2 + 1)(1+y^2) & (x^3 + x)2y  \\
0 & 1/2  
\end{bmatrix}
\end{align}
and apply (\ref{oldbound}) to obtain 
\begin{align}
\int \log_2|\frac{(1+y^2)}{2}(3x^2 + 1)|dQ(x,y) \leq C.
\end{align}
Note that there is a factor of $1/2$ coming from the stable second component in the integrand. It seems sensible that the bound should hold without this factor, as the coding and control policy need not be concerned with the stochastically stable component. Indeed, the result in this paper establishes that the above bounds holds when removing the factor of $1/2$ and is therefore a strict refinement for certain systems. We now move on to a literature review, followed by our main result and its proof. 
\end{example}

\section{Literature Review and Information Requirements for Stochastic Stability}
The presence of real-world control problems where perfect and instantaneous state information is not necessarily available to a controller has motivated the field of control under communication constraints. In this field, one wishes to study if and how it is possible to accomplish a control task under varying degrees of imperfect information. A ubiquitous problem in the field is to characterize minimum data rates required to stabilize a dynamical system. This problem has been considered extensively for linear deterministic and stochastic systems, for which one can usually characterize the minimum data rate required for closed-loop stability as the log-sum of the unstable open-loop eigenvalues. 

Some related earlier papers considering the linear case include \cite{hespanha2002towards}, \cite{delchamps1990stabilizing}, \cite{Elia2001}, \cite{nair2004stabilizability}, \cite{tatikonda2004control}, \cite{wong1999systems} and \cite{banbas89}. More recent contributions include \cite{MatveevSavkin,SahaiParts,minero2012stabilization,kostina2018exact,Martins2008,NairFagnaniSurvey}, and \cite{YukTAC2010,YukMeynTAC2010,johnston2013stochastic,YukselAMSITArxiv,YukselBasarBook} where this latter group of papers presented necessary and sufficient conditions for stability criteria such as existence of invariant measures, positive Harris recurrence and (asymptotic) ergodicity. There has been a separate line of work for the special Gaussian channel setup, which we do not review in this paper.
 
For non-linear systems, however, the majority of papers have focused on deterministic systems. Some early works include \cite{liberzon2005stabilization}, where it was established that global asymptotic stabilization of a non-linear continuous time system is feasible provided that data rates exceed a quantity related to system dimension and a Lipschitz constant, and \cite{de2005n} where non-linear feed-forward systems were considered. In \cite{nair2004topological} the authors presented the first systematic approach for determining minimal data rates for stabilization and introduced the notion of topological feedback entropy, a notion inspired by the classical open cover definition of topological entropy in dynamical systems due to Adler et al. \cite{adler1965topological}. It was established in \cite{nair2004topological} that a necessary and sufficient condition for stabilization to a compact set is the condition that the data rate in the control loop exceeds the topological feedback entropy. For the same stabilization problem, invariance entropy was introduced in \cite{colonius2009invariance}. This notion serves as a way to quantify the difficulty of a control task through the minimum number of open loop control sequences required to achieve it. The monograph \cite{kawan2013invariance} provides a detailed account of the applications of invariance entropy in determining minimum data rates, particularly for continuous time (non-linear) systems. In \cite{colonius2013note}, it was further established that under a strong invariance condition, the notions of topological feedback entropy and invariance entropy coincide in the discrete time case. A recent related development was the introduction of metric invariance entropy in \cite{colonius2018metric}. Many more interesting results have been obtained under a wealth of setting, and we refer the reader to \cite{nair2007feedback} and \cite{matveev2009estimation} for a more detailed overview of the literature. 

To the best of our knowledge, the first converse result on channel capacity for non-linear stochastic systems was established in \cite{yuksel2016stationary} using information theoretic methods. The paper provided lower bounds on channel capacity necessary for stochastic stabilization of discrete time non-linear systems over both noisy and noiseless channels for stability notions of ergodicity and entropy growth conditions. With a fundamentally different approach via stochastic growth properties, for the ergodic case a similar result was established in \cite{garcia2021ergodicity}, which relied instead on stabilization entropy. This notion, introduced in \cite{kawan2019invariance}, was a modification of invariance entropy for the stochastic case and was first used to obtain lower bounds on channel capacity required for AMS stability. The paper at hand builds on the techniques involving stabilization entropy, and provides a refinement for the lower bound in \cite[Theorem 4.1]{garcia2021ergodicity} and \cite[Theorem 4.2]{yuksel2016stationary} for the stability notion of (asymptotic) ergodicity. In particular, the statement of the result in this paper resembles \cite[Theorem 4.1]{garcia2021ergodicity} and a similar approach is used. Our main result is a strict refinement, and the proof requires a modification of stabilization entropy.

\section{Result}
Consider a subset $p \subseteq \{1,..,N\}$ of indices listed in increasing order as $p_1 < p_2 < \cdot \cdot \cdot < p_{|p|}$. Let $z_1 < \cdot \cdot \cdot < z_{N-|p|}$ denote the elements in $\{1,..,N\} \setminus p$. We define the permutation $\psi_p : \R^N \to \R^N$ by
\begin{align*}
\psi_p(x_1,...,x_N)_i = 
    \begin{cases} 
      x_{p_i} & i \leq |p| \\
      x_{{z}_{i-|p|}} & i > |p| 
    \end{cases}
\end{align*}
for $i \in \{1,..,N\}$. Also, let $\pi_p : \R^N \to \R^{|p|}$ denote the natural projection of coordinates $p_1,..,p_{|p|}$. For a map $f:\R^N \to \R^N$, a set $p$ as above, and a fixed vector $(y_1,..,y_{N-|p|})$ we define the map $f^p(\cdot,y_1,..,y_{N-|p|}):\R^{|p|} \to \R^{|p|}$ by
\begin{align}
f^p(x,y_1,..,y_{N-|p|}) \coloneqq \pi_p(f(\psi_p^{-1}(x,y_1,..,y_{N-|p|})))
\end{align}
where $x \in \R^{|p|}$. As an example, consider $N = 4, p = \{2,4\}$, a fixed vector $(y_1,y_2)$, and a function $f:\R^4 \to \R^4$ written as $f = (f_1,f_2,f_3,f_4)$ for maps $f_i:\R^4 \to R$. Then
\begin{align*}
f^p(x_1,x_2,y_1,y_2) = (f_2(y_1,x_1,y_2,x_2),f_4(y_1,x_1,y_2,x_2)). 
\end{align*}
This notation allows us to precisely state our main result. Consider the system%
\begin{align} \label{mainsys}
  x_{t+1} = f(x_t,w_t) + Bu_t%
\end{align}
where $x_t$ is $\R^N$-valued for some $N\in\N$, $B \in \R^{N \times N'}$, $u_t$ is $\R^{N'}$-valued, and $w_t$ takes values in a standard probability space $\sW$. For a fixed $w \in \sW$, let us denote the map $x \mapsto f(x,w)$ by $f_w$. Suppose that the following holds:%
\begin{itemize}
\item[(i)] The state evolution map $f$ is Borel measurable.%
\item[(ii)] The noise process $(w_t)_{t\in\N}$ is i.i.d. By abuse of notation, $\nu$ denotes both the i.i.d. measure on $\sB(\sW)$ and the noise process measure on $\sB(\sW^{\N})$.%
\item[(iii)] The map $f_w(\cdot):\R^N \rightarrow \R^N$ is $C^1$ and injective for any $w \in \sW$.%
\item[(iv)] The initial state $x_0 \in \R^N$ is random and independent of the noise process, and its law $\pi_0$ admits a bounded density.
\item[(v)] The set $\Gamma = \bigg\{p \subseteq \{1,..,N\}: \exists c_p >0 \text{ such that } |\det Df^p_w(x_{p_1},..,x_{p_{|p|}},x_{z_1},..,x_{z_{N-|p|}})| > c_p, \forall x \in \R^N, w \in \sW \bigg\}$ is non-empty.
\item[(vi)] Let $\{t_1,..,t_s\}  \in \Gamma$ be arbitrary and write the random initial state as $x_0 = (x_0^1,..,x_0^N)$. Then there exists a set $S$ consisting of possible realizations of the initial state components not indexed by $\Gamma$ such that the event that these non-indexed initial states take a realization in $S$ has non-zero probability, and the law of $(x_0^{t_1},..,x_0^{t_s})$ admits a bounded density when conditioned on the event that the non-indexed initial states take on a realization $\hat{x}$, for any $\hat{x} \in S$.
\end{itemize}

\begin{theorem} \label{NewTheorem1}
Consider system \eqref{mainsys} satisfying assumptions (i)--(vi), controlled over a noiseless channel with finite alphabet $\sM$ and capacity $C \coloneqq \log_2|\sM|$. if there exists a coding and control policy which renders the state process $(x_t)_{t \in \N}$ AMS ergodic (asymptotically ergodic) with asymptotic mean $Q$, then we must have that% 
\begin{align} \label{capacitybound}
  \max_{p \in \Gamma } \int \int \log|\det Df^p_w(x_{p_1},..,x_{p_{|p|}},x_{z_1},..,x_{z_{N-|p|}})| \\ 
  dQ(x_1,..,x_N) d \nu(w) \leq C. \nonumber
\end{align}
\end{theorem}
where the Jacobian above is the $|p| \times |p|$ matrix of partial derivatives of $f^p_w(\cdot,x_{z_1},..,x_{z_{N-|p|}})$ evaluated at $(x_{p_1},..,x_{p_{|p|}}) \in \R^{|p|}$. 
\begin{remark}
Observe that by taking $p = \{1,..,M\}$ (if $\{1,..,M\} \in \Gamma$), we recover the bound (\ref{oldbound}) established previously in \cite{garcia2021ergodicity} and \cite{yuksel2016stationary}. For a large class of systems however, it is clear that Theorem \ref{NewTheorem1} is a strict refinement, as can be seen by noting that in Example \ref{ex2bd}, taking $p = \{1\}$ recovers the sharper bound
\begin{align}
\int \log_2|(1+y^2)(3x^2 + 1)|dQ(x,y) \leq C.
\end{align}
It is clear that for linear system, the new bound recovers the tight linear bound such as in Example \ref{ex1bd}.
\end{remark}
\begin{remark}
Note also that the technical assumption (vi) is satisfied if the initial state has independent components each admitting a bounded density. Suppose now that for a given system, the assumptions (i)-(v) of Theorem \ref{NewTheorem1} are satisfied, but (vi) only hold for certain subsets of $\Gamma$. Then the theorem will still hold, however the $\max$ in (\ref{capacitybound}) should be taken only over subsets of $\Gamma$ for which the assumption (vi) holds. This last observation will become clear from the proof. 
\end{remark}
\begin{remark}
Noting that the Jacobian determinant is invariant under a linear change of coordinates, we note that the bound in (\ref{NewTheorem1}) is invariant under a linear change of coordinates. Note however that assumption (v) is not coordinate independent; it is not hard to see that for certain systems, the choice of coordinates may result in a different (or empty) set $\Gamma$. Under a non-linear coordinate change, it is not clear if the above bound is invariant (or if control even remains additive), thus a possible future research direction is to consider the problem of optimizing the coordinate system chosen in order to maximize the bound. 
\end{remark}

%\end{example}

\section{Proofs}
We first fix an integer $m \leq M$ and view the map $f_w$ as a function of two vectors, i.e. we decompose the state into a pair $(x,y)$ where $x \in \R^m$ and $y \in \R^{N-m}$. Consider the control system in Theorem \ref{NewTheorem1} and note that for a fixed sequence of controls $u \coloneqq (u_t)_{t \in \N}$, a fixed sequence of noise symbols $w \coloneqq (w_t)_{t \in \N}$, and a fixed initial state $x_0 \in \R^N$, the state process $(x_t)_{t \in \N}$ is deterministic. Let us introduce the notation $\phi(t,x_0,u,w) \coloneqq x_t$ for every $t \in \N$. Letting $\pi_m$ and $\pi_{N-m}$ denote the natural projection of $\R^N$ on to the first $m$ and last $N-m$ coordinates respectively, we further define $\phi^m(t,x_0,u,w) \coloneqq \pi_m(x_t)$ and $\phi^{N-m}(t,x_0,u,w) \coloneqq \pi_{N-m}(x_t)$ so that
\begin{align*}
 \phi(t,x_0,u,w) = (\phi^m(t,x_0,u,w),\phi^{N-m}(t,x_0,u,w)) \text{   for every   }t \in \N.
\end{align*}
We now provide a sketch of the proof, which relies on the notion of stabilization entropy and an associated lemma relating it to channel capacity. Compared to \cite{garcia2021ergodicity}, we consider a version of stabilization entropy with an additional collection of sets since we are decomposing the state space into two components. The definition follows:
\subsection{Stabilization Entropy}
\begin{definition}  \label{spanningdef}
Let $(D_j)_{j = 1}^{j = d} \subseteq \sB(\R^m)$, $(E_k)_{k=1}^{k=e} \in \sB(\R^{N-m})$ and $(F_l)_{l=1}^{f} \in \sB(\sW)$ be finite disjoint unions of Borel sets and define
\begin{align*}
D \coloneqq \bigsqcup_{j=1}^{d}D_j \quad \quad E \coloneqq \bigsqcup_{k=1}^{e}E_k \quad \quad F \coloneqq \bigsqcup_{l=1}^{f}F_l.
\end{align*}
Let also $R$ denote a collection of numbers $r_{j,k,l} \in [0,1]$ for $j \in \{1,\ldots,d\}$, $k \in \{1,\ldots,e\}$ and $l \in \{1,\ldots,f\}$ satisfying%
\begin{equation*}
  1 - r := \sum_{j=1}^d \sum_{k=1}^e \sum_{l=1}^f (1 - r_{j,k,l}) \in [0,1]%
\end{equation*}
and fix $T \in \N$ and $\rho \in (0,1)$. A set $S \subseteq (\R^N)^T$ of control sequences of length $T$ is called $(T,D,E,F,\rho,R)$-spanning iff there exists $\tilde{\Omega} \in \sF$ such that the following conditions:
\begin{itemize}
\item $P(\tilde{\Omega}) \geq 1-\rho$.%
\item For each $\omega \in \tilde{\Omega}$, there exists a control sequence $u \in S$ such that
\begin{align*}
  \frac{1}{T}|\{t \in [0;T-1] : (\phi^m(t,x_0(\omega),u,w(\omega)),\phi^{N-m}(t,x_0(\omega),u,  \\
  w(\omega)),w_t(\omega)) \in D_j \times E_k \times F_l \}| \geq 1 - r_{j,k,l}%
\end{align*}
for all $j$, $k$ and $l$.
\end{itemize}
both hold.
\end{definition}

We slightly abuse notation writing $(T,D,E,F,\rho,R)$-spanning instead of \\ $(T,(D_j)_{j=1}^{d},(E_k)_{k=1}^{e},(F_l)_{l=1}^f,\rho,R)$-spanning. Whenever we do this however, the specific sequences of sets making up the disjoint unions will be clear from context. We will use the size of spanning sets to quantify the difficulty of a control task. This leads to:

\begin{definition} \label{thm;defofsets}
(Stabilization Entropy) For the system (\ref{mainsys}), and sequences of sets as in Definition \ref{spanningdef}, we define the $(D,E,F,\rho,R)$-stabilization entropy by%
\begin{equation*}
  h(D,E,F,\rho,R) := \limsup_{T \to \infty}\frac{1}{T} \log s(T,D,E,F,\rho,R),%
\end{equation*}
where $s(T,D,E,F,\rho,R)$ denotes the smallest cardinality of a $(T,D,E,F,\rho,R)$-spanning set. We define this quantity to be $\infty$ if no finite spanning set exists.%
\end{definition}

Finite $(T,D,E,F,\rho,R)$-spanning sets need not exist in general but as we will shortly see, they exist in desired scenarios. The following lemma relates stabilization entropy with channel capacity.

\begin{lemma} \label{myNewLemma}
Consider system (\ref{mainsys}) with the assumptions of Theorem \ref{NewTheorem1} (i.e., a coding and control policy exists over a noiseless channel of capacity $C = \log_2|\sM|$ which makes the state process AMS ergodic with asymptotic mean $Q$). Let $D,E$ and $F$ be as in Definition \ref{spanningdef} and let $\rho \in (0,1)$ be arbitrary. Let $\epsilon > 0$ and define the collection of numbers $R_\epsilon \coloneqq (r_{j,k,l})_{1 \leq j \leq d,1 \leq k \leq e, 1 \leq l \leq f}$, where%
\begin{align*}
  r_{j,k,l} := \begin{cases}
      (1+\epsilon)(1 - Q(D_j \times E_k)\nu(F_l)) & \text{ if } \kappa_{j,k,l} \in (0,1) \\
      1 & \text{ if } \kappa_{j,k,l} = 0  \\
      \epsilon & \text{ if } \kappa_{j,k,l} = 1
    \end{cases}
\end{align*}
where we use the shorthand $\kappa_{j,k,l} \coloneqq Q(D_j\times E_k)\nu(F_l)$. Although the $r_{j,k,l}$'s are $\epsilon$-dependent, we suppress this from the notation. The claim of the lemma is that for all sufficiently small $\epsilon>0$, the stabilization entropy is well defined and satisfies
\begin{equation}\label{eq_centbound}
  h(D,E,F,\rho,R_\epsilon) \leq C.%
\end{equation}
\end{lemma}

\begin{proof}
We note that for $\epsilon > 0$ sufficiently enough, the conditions%
\begin{enumerate}[(i)]
\item $1 - r \coloneqq \sum_{j,k,l} (1 - r_{j,k,l}) \in [0,1]$,%
\item $1 - (1 + \epsilon)(1 - Q(D_j,E_k)\nu(F_l)) \in (0,1) $ for all $j,k,l$ with $Q(D_j,E_k)\nu(F_l) \in (0,1)$,%
\end{enumerate}
are both satisfied, thus ensuring that for such a small $\epsilon$ the stabilization entropy $h(D,E,F,\rho,R_\epsilon)$ is well defined. Consider system (\ref{mainsys}) evolving according to the fixed coding and control policy which renders the state process $(x_t)_{t \in \N}$ AMS ergodic with AMS mean $Q$. To prove that inequality we consider three cases:

\textbf{Case 1:} We first consider the case where $Q(D_j \times E_k)\nu(F_l) \in (0,1)$ for all $j,k,l$. Let $\epsilon > 0$ be small enough such that $\epsilon < \rho$ as well as conditions (i) and (ii) are satisfied. We will show that for any such $\epsilon$ the claim holds. Let us denote the process measure by $\mu$, which is AMS by assumption. Now, for any $V \in \sB(\sW)$, it is clear by the i.i.d.~property that%
\begin{equation*}
  P\Bigl(\Bigl\{\omega \in \Omega : \lim_{T \to \infty} \frac{1}{T} \sum_{t=0}^{T-1}\1_V(w_t(\omega)) = \nu(V)\Bigr\}\Bigr) = 1.%
\end{equation*}
Noting that $x_t$ and $w_t$ are independent at each time step $t$, $(w_t)_{t \in N}$ is i.i.d, and recalling equation (\ref{spce}), it follows that $P(\hat{\Omega}) = 1$ where 
\begin{align}\label{ergodicbehaviour}
  \hat{\Omega} &\coloneqq \{\omega \in \Omega : \lim_{T \to \infty} \frac{1}{T} \sum_{t=0}^{T-1} \1_{D_j}(\pi_m(x_t(\omega))) \cdot \\ &\1_{E_k}(\pi_{N-m}(x_t(\omega))) \cdot \1_{F_l}(w_t(\omega)) = Q(D_j \times E_k)\nu(F_l),\ \forall j,k,l\Bigr\} \nonumber
\end{align}
where we note that the above set can be written as the intersection of a finite number of full measure sets. We continue by defining the events%
\begin{align*}
  &E_I^J \coloneqq \Bigl\{\omega \in \Omega :\ \Bigl| \frac{1}{T} \sum_{t=0}^{T-1} \1_{D_j}(\pi_m(x_t(\omega)))\1_{E_k}(\pi_{N-m}(x_t(\omega))) \\ &\1_{F_l}(w_t(\omega)) -Q(D_j \times E_k)\nu(F_l) \Bigr| < \frac{1}{I} \forall j,k,l \text{ whenever } T \geq J \Bigr\}
\end{align*}
and note that for any $I \in \N$, it is clear that $ \hat{\Omega} \subseteq \bigcup_{J = 1}^{\infty}E_I^J$ therefore $P \Big( \bigcup_{J = 1}^{\infty}E_I^J \Big) = 1$. 
Let now $I_0$ be large enough such that%
\begin{equation*}
  \frac{1}{I_0} \leq \epsilon (1 - Q(D_j \times E_k)\nu(F_l))  \mbox{\quad for all\ } j,k,l%
\end{equation*}
and observe that $E_{I_0}^1 \subseteq E_{I_0}^2 \subseteq E_{I_0}^3 \subseteq \cdots$. By continuity of probability, we have%
\begin{equation*}
  \lim_{J \rightarrow \infty} P(E_{{I_0}}^{J}) = P\bigg(\bigcup_{J = 1}^{\infty}E_{I_0}^J\bigg) = 1,%
\end{equation*}
and thus there exists $J_0 \in \N$ such that $P(E_{I_0}^{J}) \geq 1 - \epsilon$ for all $J \geq J_0$. For an arbitrary $T \geq J_0$, we define the set of control sequences%
\begin{equation*}
  S_T := \{u_{[0;T-1]}(\omega) : \omega \in E_{I_0}^T\}.%
\end{equation*}
We claim that this set is $(T,D,E,F,\rho,R_\epsilon)$-spanning. We use the set $\tilde{\Omega}_T := E_{I_0}^T \in \sF$ to show this, where we note that $P(\tilde{\Omega}_T) \geq 1 - \epsilon > 1 - \rho$, satisfying the first requirement of the spanning set definition (Definition \ref{spanningdef}). To check the second condition, observe that for every $\omega \in \tilde{\Omega}_T$ and every triple $j,k,l$, the control sequence $u_{[0;T-1]}(\omega) \in S_T$ results in the joint state-noise process satisfying%
\begin{align*}
  \Big|\frac{1}{T} \sum_{t=0}^{T-1} \1_{D_j}(\pi_m(x_t(\omega)))\1_{E_k}((\pi_{N-m}(x_t(\omega))) \1_{F_l}(w_t(\omega)) - \\ Q(D_j \times E_k)\nu(F_l)\Big|  
  < \frac{1}{I_0} \leq \epsilon (1 - Q(D_j \times E_k)\nu(F_l)).%
\end{align*}
which implies that
\begin{align*}  
& \frac{1}{T}|\{t \in [0;T-1] : (\phi(t,x_0(\omega),u_{[0;T-1]}(\omega),w(\omega)),w_t(\omega)) \in \\ &D_j \times E_k \times F_l \}|
   \geq 1 - (1+\epsilon)(1 - Q(D_j,E_k)\nu(F_l)) = 1 - r_{j,k,l}%
\end{align*}
which establishes the second condition, since the triple $j,k,l$ was arbitrary. We have thus established that $S_T$ is $(T,D,E,F,\rho,R_\epsilon)$-spanning.
Since the fixed causal coding and control policy can generate at most $|\sM|^T$ distinct control sequences by time $T$, it follows that $|S_T| \leq |\sM|^T$, therefore $s(T,D,E,F,\rho,R_\epsilon) \leq |\sM|^T$. Recalling that $T \geq J_0$ was arbitrary, we find that%
\begin{equation*}
  \log s(T,B,D,\rho,R_\epsilon) \leq T \log_2|\sM| = TC \mbox{\quad for all\ } T \geq J_0,%
\end{equation*}
and therefore dividing by $T$ and letting $T\rightarrow\infty$ yields the desired capacity bound \eqref{eq_centbound}, completing the proof for Case 1. 

\textbf{Case 2:} We now consider the case where every triple of sets $(D_j,E_k,F_l)$ satisfies $Q(D_j \times E_k)\nu(F_l) \in [0,1)$. Suppose that $j,k,l$, is such that $Q(D_j \times E_k)\nu(F_l) = 0$. Then $1 - r_{j,k,l} = 0$ and the second condition in Definition \ref{spanningdef} is vacuously satisfied. Combining this with Case 1, the result follows.    

\textbf{Case 3:} Finally, we consider the case where for some indices $j,k,l$, $Q(D_j \times E_k)\nu(F_l) = 1$. Because each collection of sets is disjoint, $Q(D_{j'} \times E_{k'})\nu(F_{l'}) = 0$ whenever $(j,k,l) \neq (j',k',l')$. The analysis reduces to establishing the second condition in Definition \ref{spanningdef} for the single set $D_j \times E_k \times F_l$ with $Q(D_j \times E_k)\nu(F_l) = 1$. Using an almost identical argument as in Case 1, the result follows. Alternatively, the analysis of a single set can be found in \cite{kawan2019invariance}, where AMS was considered instead of AMS ergodicity as the control objective. Since AMS ergodicity implies AMS, and $h(D,E,F,\rho,R)$ reduces to the stabilization entropy notion used in \cite{kawan2019invariance} in case of a single set, the desired inequality follows.%
\end{proof}

\subsection{Proof of Theorem \ref{NewTheorem1}}
To prove Theorem \ref{NewTheorem1}, we will approximate the integral in equation (\ref{capacitybound}) form below using simple functions. We will prove that each of these approximations is upper bounded by the stabilization entropy, which in turn is no larger than the channel capacity. Taking a limit will yield the result. First, a few simplifications are in order.

\begin{proof}

Recalling that we fixed an integer $m \leq N$, we define $p \coloneqq \{1,..,m\}$. WLOG, it suffices to establish
\begin{align}\label{modifiedineq}
  \int \int \log_2|\det Df^p_w(x_{p_1},...,x_{p_{m}},x_{z_1},...,x_{z_{N-m}})| \nonumber \\ d Q(x_1,...,x_N) d \nu(w) \leq C%
\end{align}
since by a relabeling of coordinates, any other set $p' \in \Gamma$ can be written in the form $\{1,2,3,..,|p'|\}$. 

Recall now that by assumption (vi) in Theorem \ref{NewTheorem1}, there exists a set $S \subseteq \R^{N-m}$ with positive probability in the sense that
\begin{align*}
P(\{\omega \in \Omega : (x_0^{m+1}(\omega),x_0^{m+2}(\omega),...,x_0^{N}(\omega)) \in S \} ) > 0.
\end{align*}
Moreover, the set $S$ has the property that for any $\hat{x} \in S$, conditioning on the event $\{\omega \in \Omega : (x_0^{m+1},..,x_0^{N-m})(\omega) = \hat{x} \}$ results in the law of the random vector $(x_0^1,..,x_0^m)$ admitting a bounded density. Let $\pi_0'$ denote this conditional law. We now establish inequality (\ref{modifiedineq}) under slightly different assumptions than those of Theorem \ref{NewTheorem1}. More specifically, we impose that
\begin{itemize}
\item The last $N - m$ components of the initial state are deterministic, taking the value $\hat{x}$ for some arbitrary $\hat{x} \in S$.
\item The initial $m$ components of $x_0$ are distributed according to the law $\pi_0'$.
\end{itemize}
We claim that if we can establish (\ref{modifiedineq}) under these modified assumptions, the inequality will also hold under the assumptions of Theorem \ref{NewTheorem1}. To see this, suppose otherwise. Then a coding and control policy exists which stabilizes the system in Theorem \ref{NewTheorem1} over a channel of capacity strictly less than the LHS of (\ref{modifiedineq}). Since the stabilizing scheme works almost surely, this coding and control policy would also stabilize the system under the modified assumptions above for at least one $\hat{x} \in S$ (since $S$ has non-zero measure), resulting in a contradiction since we are assuming that (\ref{modifiedineq}) holds for the modified system and for every $\hat{x}\in S$. We now proceed under the modified assumptions, and redefine $\pi_0 \coloneqq \pi_0'$ to refer to the conditional law of $\pi_m(x_0)$.

Let $c \in (0,1)$ be such that $c < |\det Df^p_w(x,y)|$ for all $(x,y) \in \R^m \times \R^{N-m}$ and $w \in \sW$. Let also $\delta > 0$ (think of this as small) and $\rho \in (0,1)$ (think of this as close to $1$) be arbitrary. Next, fix Borel sets $D \subset \R^m$ and $E \subset \R^{N-m}$ satisfying that $D \times E$ have finite $N$-dimensional Lebesgue measure and that
\begin{equation*}
  Q(D \times E) > 1 - \frac{\delta}{2|\log c|}%
\end{equation*}
holds (such sets can easily be found due to continuity of probability). Put also $F = \sW$ and let $(D_j)_{j=1}^{d}$, $(E_k)_{k=1}^{e}$ and $(F_l)_{l=1}^{f}$ be (disjoint) partitions of $D$, $E$ and $F$ respectively. Let now $\epsilon > 0$ be small enough so that Lemma \ref{myNewLemma} holds, resulting in%
\begin{equation*} 
  h(D,E,F,\rho,R_\epsilon) \leq C,%
\end{equation*}
where $R_\epsilon$ is the associated collection of $r_{j,k,l}$'s as defined in Lemma \ref{myNewLemma}. Let also $1 - r \coloneqq \sum (1-r_{j,k,l})$. Expanding out, it is easy to see (recalling that $\nu(F) = 1$) that $r = 1 - (1 + \epsilon)Q(D \times E) + def \epsilon$ (or $r = \epsilon$ if one of the $D_j \times E_k \times F_l$'s has full $Q \times \nu$-measure) thus we see that for every sufficiently small $\epsilon$,%
\begin{equation}\label{eq_r_choice}
  2r < \frac{\delta}{|\log c|}.%
\end{equation}
Now fix a sufficiently large $T\in\N$ and let $S_T$ be a finite $(T,D,E,F,\rho,R_\epsilon)$-spanning set (whose existence is guaranteed by the proof of Lemma \ref{myNewLemma}) with $\tilde{\Omega}_T \in \sF$, $P(\tilde{\Omega}) \geq 1 - \rho$, the associated subset of $\Omega$. Letting $x_0^m$ denote the vector consisting of the first $m$ components of $x_0$, we proceed by defining%
\begin{align*}
A &\coloneqq \{(w(\omega),x_0^m(\omega)) : \omega \in \tilde{\Omega}\}, \\
A(u) &\coloneqq \{ (w,x) \in \sW^{\N} \times \R^m : \forall j,k,l, \\ &\frac{1}{T}\sum_{t=0}^{T-1}\1_{D_j \times E_k \times F_l}(\phi(t,(x,\hat{x}),u,w),w_t) \geq 1 - r_{j,k,l}  \}. \\
A(u,w) &\coloneqq \{ x \in \R^m : (w,x) \in A(u) \}.
\end{align*}
Letting $m$ denote the $m$-dimensional Lebesgue measure, we see that
\begin{align}\label{eq_nc1}
  A \subseteq \bigcup_{u \in S_T}A(u), \quad (\nu \times m)(A(u)) = \int m(A(u,w)) d \nu(w),%
\end{align}
where the equality follows from the Fubini-Tonelli theorem (and the containment by definition of the sets). Letting $M>0$ be an upper bound for the density of $\pi_0$, we have that
\begin{equation}\label{eq_nc2} 
  1 - \rho \leq (\nu \times \pi_0)(A) \leq M\cdot(\nu \times m)(A).%
\end{equation} 
Combining (\ref{eq_nc1}) and (\ref{eq_nc2}), we obtain the key inequality:
\begin{align} \label{eq_nc3} 
 \frac{1}{M}(1 - \rho) &\leq (\nu \times m)(A) \leq |S_T| \max_{u\in S_T} (\nu \times m)(A(u)) \nonumber \\ 
 &= |S_T| \max_{u\in S_T} \int m(A(u,w)) d \nu(w).
\end{align}
The next step in the proof is to obtain upper bounds for the volume $m(A(u,w))$. We proceed by defining a set consisting of disjoint collections of subsets of $\{0,\ldots,T-1\}$:%
\begin{align*}
  \A &\coloneqq \{ \Lambda = \{\Lambda_{j,k,l}\}_{j,k,l} : \bigsqcup_{j=1}^{d}\bigsqcup_{k=1}^{e}\bigsqcup_{l=1}^{f}\Lambda_{j,k,l} \subseteq \{0,\ldots,T-1\},\\
  & |\Lambda_{j,k,l}| \geq (1 - r_{j,k,l})T, \forall j=1,..,d,\ k=1,..,e,\ l=1,..,f \}%
\end{align*}
and note that as a consequence of the definition, $|\bigsqcup_{j=1}^{d}\bigsqcup_{k=1}^{e}\bigsqcup_{l=1}^{f}\Lambda_{j,k,l}| \geq (1-r)T$ for all $\Lambda \in \A$. We note that such sets can only be found for $T$ sufficiently large, however as we will be taking a limit as $T \to \infty$, this is not a problem. For $\Lambda \in \A$, define the set%
\begin{align*}
  A(u,w,\Lambda) := \{x \in \R^m : (\phi(t,(x,\hat{x}),u,w),w_t) \in D_j \times E_k \times F_l \Leftrightarrow \\ t \in \Lambda_{j,k,l} \text{ for all } j,k,l \}.%
\end{align*}
It is not hard to see that $A(u,w) = \bigsqcup_{\Lambda \in \A} A(u,w,\Lambda)$ is a disjoint union, thus (\ref{eq_nc3}) becomes%
\begin{align} \label{eq_nc4}
 \frac{1}{M}(1 - \rho) \leq |S_T| \max_{u\in S_T} \int \sum_{\Lambda \in \A} m(A(u,w,\Lambda)) d \nu(w).
\end{align}
Our next step is to bound the volumes of the form $m(A(u,w,\Lambda))$. Writing $\varphi_{t,u,w}(\cdot) := \varphi(t,(\cdot,\hat{x}),u,w)$ we define%
\begin{equation*}
  A_t(u,w,\Lambda) := \phi_{t,u,w}(A(u,w,\Lambda)),\quad t = 0,1,\ldots,T-1,%
\end{equation*}	
and observe that%
\begin{equation*}
  A_t(u,w,\Lambda) \subseteq D_j \text{\quad whenever}\ t \in \Lambda_{j,k,l} \ \forall \ j,k,l.%
\end{equation*}
Next, we define the following numbers:%
\begin{align*}
  c_{j,k,l} \coloneqq \inf_{(x,y,w) \in D_j \times E_k \times F_l}|\det Df^p_w(x,y)|.%
\end{align*}
Recalling that by assumption $f_w^p(\cdot,y)$ is injective and $C^1$, it follows that for all $(j,k,l)$ we have that%
\begin{align*}
  m(A_{t+1}(u,w,\Lambda)) &\geq c_{j,k,l} \cdot m(A_t(u,w,\Lambda)) \text{ whenever } t\in \Lambda_{j,k,l}, \\
  m(A_{t+1}(u,w,\Lambda)) &\geq c \cdot m(A_t(u,w,\Lambda)) \text{ whenever } t \notin \bigsqcup \Lambda_{j,k,l}.
\end{align*}
Letting $t^*(\Lambda_{j,k,l}) \coloneqq \max\Lambda_{j,k,l}$, $t^*(\Lambda) \coloneqq \max_{j,k,l}t^*(\Lambda_{j,k,l})$, applying the above inequalities repeatedly, and recalling that $c \leq c_{j,k,l}$, it is not hard to see that%
\begin{align*}
  m(A(u,w,\Lambda)) \Bigl( \prod_{j=1}^{d}\prod_{k=1}^{e}\prod_{l=1}^{f} c_{j,k,l}^{|\Lambda_{j,k,l}|-1} \Bigr) c^{rT+def} \\ \leq m(A_{t^*(\Lambda)}(u,w,\Lambda)).%
\end{align*}
where in principle, all the exponents of the $c_{j,k,l}$'s should be $|\Lambda_{j,k,l}|$, except for possibly one which should be $|\Lambda_{j,k,l}| - 1$. We do not know which one though, so we write the weaker inequality as above. Combining this with (\ref{eq_nc4}), we obtain%
\begin{align*}
  \frac{1}{M}(1 - \rho) &\leq |S| \max_{u\in S_T} \sum_{\Lambda \in \A} \int m(A_{t^*(\Lambda)}(u,w,\Lambda)) c^{-(rT+def)} \\  &\prod_{j = 1}^{d} \prod_{k = 1}^{e} \prod_{l = 1}^{f} c_{j,k,l}^{-(|\Lambda_{j,k,l}|-1)}  d \nu(w),
\end{align*}
and note that the right hand side of the above can be written as
\begin{align*}
  &= |S| \cdot c^{-(rT+def)} \max_{u\in S_T} \sum_{t_{1,1,1} = (1-r_{1,1,1})T}^{T} \cdots \sum_{t_{d,e,f} = (1-r_{d,e,f})T}^{T} \\
  &\quad \int \sum_{\Lambda \in \A :\ t^*(\Lambda_{j,k,l}) = t_{j,k,l} \forall j,k,l} m(A_{t^*(\Lambda)}(u,w,\Lambda)) \\
  &\prod_{j = 1}^{d} \prod_{k = 1}^{e} \prod_{l=1}^{f} c_{j,k,l}^{-(|\Lambda_{j,k,l}|-1)} d \nu(w)  \\ 
  &\leq |S| \cdot c^{-(2rT+def)} \max_{u\in S_T} \sum_{t_{1,1,1} = (1-r_{1,1,1})T}^{T} \cdots \sum_{t_{d,e,f} = (1-r_{d,e,f})T}^{T}\\
  &\quad \int \sum_{\Lambda \in \A :\ t^*(\Lambda_{j,k,l}) = t_{j,k,l} \forall j,k,l} m(A_{t^*(\Lambda)}(u,w,\Lambda)) \\ 
  &\prod_{j = 1}^{d} \prod_{k = 1}^{e} \prod_{l=1}^{f} c_{j,k,l}^{-((1-r_{j,k,l})T-1)} d \nu(w).% 
\end{align*}
where the last inequality follows by noting that
\begin{align*}
  &c^{rT + def} \prod_{j,k,l} c_{j,k,l}^{|\Lambda_{j,k,l}|-1} = c^{rT + \sum_{j,k,l}|\Lambda_{j,k,l}|} \prod_{j,k,l} \Bigl(\frac{c_{j,k,l}}{c}\Bigr)^{|\Lambda_{j,k,l}|-1}\\
  &\geq c^{rT + \sum_{j,k,l}|\Lambda_{j,k,l}|} \prod_{j,k,l} \Bigl(\frac{c_{j,k,l}}{c}\Bigr)^{(1-r_{j,k,l})T-1} \\
  &= c^{rT + \sum_{j,k,l}|\Lambda_{j,k,l}| - (1-r)T + def} \prod_{j,k,l}c_{j,k,l}^{(1-r_{j,k,l})T-1}\\
  &\geq c^{2rT + def} \prod_{j,k,l}c_{j,k,l}^{(1-r_{j,k,l})T-1}.%
\end{align*}
Observe that the sets $A_{t^*(\Lambda)}(u,w,\Lambda)$ with $\Lambda \in \A$, $t^*(\Lambda)$ fixed, are pairwise disjoint, since they are the images of the corresponding sets $A(u,w,\Lambda)$ under the injective map $\varphi_{t^*(\Lambda),u,w}$. Moreover, all of these sets are contained in $D$, hence%
\begin{equation*}
  \sum_{\Lambda \in \A : t^*(\Lambda_{j,k,l}) = t_{j,k,l} \forall j,k,l} m(A_{t^*(\Lambda)}(u,w,\Lambda)) \leq m(D).%
\end{equation*}
Together with the above chain of inequalities, this implies%
\begin{align*}
 \frac{1}{M}(1 - \rho) \leq  |S_T| \cdot m(D) \cdot c^{-(2rT+def)} \cdot \\ \prod_{j = 1}^{d}  \prod_{k = 1}^{e} \prod_{l=1}^{f} c_{j,k,l}^{-((1-r_{j,k,l})T-1)}  \prod_{j=1}^d \prod_{k=1}^e \prod_{l=1}^f (r_{j,k,l}T + 1).% 
\end{align*}
Since this inequality holds for every $T$ sufficiently large, we can take logarithms on both sides, divide by $T$ and let $T \rightarrow \infty$. This results in%
\begin{equation*}
  0 \leq h(D,E,F,\rho,R_{\epsilon}) - 2r \log c - \sum_{j=1}^d \sum_{k=1}^e \sum_{l=1}^f (1-r_{j,k,l}) \log c_{j,k,l}.%
\end{equation*}
Recalling the definition of $r_{j,k,l}$, the fact that $\epsilon$ can be chosen arbitrarily small and \eqref{eq_r_choice}, this leads to the estimate%
\begin{align*}
  C + \delta \geq \sum_{j=1}^d \sum_{k=1}^e \sum_{l=1}^f Q(D_j \times E_k)\nu(F_l) \\ \inf_{(x,y,w) \in D_j \times E_k \times F_l}\log |\det Df^p_w(x,y)|.%
\end{align*}
Considering the supremum of the right-hand side over all finite measurable partitions of $D,E$ and $F = \sW$ leads to%
\begin{align*}
  C + \delta \geq \int \int \1_{D \times E}(x_1,..,x_N) \log |\det Df^p_w(x_1,..,x_N)| \\ dQ(x_1,..,x_N) d \nu(w),%
\end{align*}
where we use that the integrand is uniformly bounded below by $\log c$ (and hence, we can assume that it is non-negative). Considering now an increasing sequence of sets $D_k \times E_k \subset \R^N$ whose union is $\R^N$, we can invoke the theorem of monotone convergence to obtain the desired estimate, observing that $\delta$ can be made arbitrarily small as $D_k \times E_k$ becomes arbitrarily large. This completes the proof.
\end{proof}

\addtolength{\textheight}{-3cm}

\section{CONCLUSIONS}
In conclusion we have - for a certain class of non-linear systems - established a sharper bound on channel capacity required for ergodic stabilization. The techniques involved in the proof are stabilization entropy, a volume growth argument, and the property that almost surely, system sample paths visit regions of the state space at a frequency given by an ergodic measure. There are three possible avenues of further investigation. First, it would be interesting to enlarge the class of noise processes for which the bounds in this paper hold. We have considered only i.i.d. noise, however it is possible that an ergodic-like property (i.e. that equation (\ref{ergodicbehaviour}) holds) for the joint state-noise process will hold for less restrictive classes of noise. Secondly, it seems possible to attempt the generalization of this result for the noisy channel case. Using stabilization entropy techniques, \cite{garcia2021ergodicity} established the bound (\ref{oldbound}) for scalar systems controlled over Discrete Memoryless Channels. Given that sharper bounds can be established for multi-dimensional systems, it seems worthwhile to attempt to generalize the one dimensional noisy-channel result to many dimensions, and combine it with the arguments in this paper to sharpen the bound. At a first glance, there appear to be no significant technical challenges to overcome. Finally, we note that using the current method, it is not possible to establish the bound for the most general class of systems of the form $x_{t+1} = f(x_t,w_t,u_t)$. The reason for this is that the proof in this paper relies heavily on the fact that with additive control, volume growth of the map $x \mapsto f(x,w)+u$ does not depend on the choice of $u$. A future direction of investigation is to consider the most general class of systems, and impose restrictions on the type of causal coding and control policies in such a way so as to ensure that the control process has ergodic properties. We conclude by noting that we were unable to obtain the sharper bound established in this paper using information theoretic methods, with the main impediment being the fact that when splitting the state and conditioning on a past state realization, the unstable and stable state components may not be independent random variables, which is required for the information theoretic methods to apply. 

\bibliographystyle{IEEEtran}
\bibliography{IEEEabrv,myreferences,SerdarBibliography}

\begin{thebibliography}{10}
\providecommand{\url}[1]{#1}
\csname url@samestyle\endcsname
\providecommand{\newblock}{\relax}
\providecommand{\bibinfo}[2]{#2}
\providecommand{\BIBentrySTDinterwordspacing}{\spaceskip=0pt\relax}
\providecommand{\BIBentryALTinterwordstretchfactor}{4}
\providecommand{\BIBentryALTinterwordspacing}{\spaceskip=\fontdimen2\font plus
\BIBentryALTinterwordstretchfactor\fontdimen3\font minus
  \fontdimen4\font\relax}
\providecommand{\BIBforeignlanguage}[2]{{%
\expandafter\ifx\csname l@#1\endcsname\relax
\typeout{** WARNING: IEEEtran.bst: No hyphenation pattern has been}%
\typeout{** loaded for the language `#1'. Using the pattern for}%
\typeout{** the default language instead.}%
\else
\language=\csname l@#1\endcsname
\fi
#2}}
\providecommand{\BIBdecl}{\relax}
\BIBdecl

\bibitem{walters2000introduction}
P.~Walters, \emph{An introduction to ergodic theory}.\hskip 1em plus 0.5em
  minus 0.4em\relax Springer Science \& Business Media, 2000, vol.~79.

\bibitem{gray2009probability}
R.~M. Gray, \emph{Probability, random processes, and ergodic properties}.\hskip
  1em plus 0.5em minus 0.4em\relax Springer, 2009.

\bibitem{garcia2021ergodicity}
N.~Garcia, C.~Kawan, and S.~Y\"uksel, ``Ergodicity conditions for controlled
  stochastic nonlinear systems under information constraints: A volume growth
  approach,'' \emph{SIAM Journal on Control and Optimization}, vol.~59, no.~1,
  pp. 534--560, 2021.

\bibitem{yuksel2016stationary}
S.~Y\"uksel, ``Stationary and ergodic properties of stochastic nonlinear
  systems controlled over communication channels,'' \emph{SIAM Journal on
  Control and Optimization}, vol.~54, no.~5, pp. 2844--2871, 2016.

\bibitem{hespanha2002towards}
J.~Hespanha, A.~Ortega, and L.~Vasudevan, ``Towards the control of linear
  systems with minimum bit-rate,'' in \emph{Proc. of the Int. Symp. on the
  Mathematical Theory of Networks and Syst}.\hskip 1em plus 0.5em minus
  0.4em\relax Citeseer, 2002, p.~1.

\bibitem{delchamps1990stabilizing}
D.~F. Delchamps, ``Stabilizing a linear system with quantized state feedback,''
  \emph{IEEE transactions on automatic control}, vol.~35, no.~8, pp. 916--924,
  1990.

\bibitem{Elia2001}
N.~Elia and S.~K. Mitter, ``Stabilization of linear systems with limited
  information,'' \emph{IEEE Transactions on Automatic Control}, vol.~46, no.~9,
  pp. 1384--1400, 2001.

\bibitem{nair2004stabilizability}
G.~N. Nair and R.~J. Evans, ``Stabilizability of stochastic linear systems with
  finite feedback data rates,'' \emph{SIAM Journal on Control and
  Optimization}, vol.~43, no.~2, pp. 413--436, 2004.

\bibitem{tatikonda2004control}
S.~Tatikonda and S.~Mitter, ``Control under communication constraints,''
  \emph{IEEE Transactions on Automatic Control}, vol.~49, no.~7, pp.
  1056--1068, 2004.

\bibitem{wong1999systems}
W.~S. Wong and R.~W. Brockett, ``Systems with finite communication bandwidth
  constraints. ii. stabilization with limited information feedback,''
  \emph{IEEE Transactions on Automatic Control}, vol.~44, no.~5, pp.
  1049--1053, 1999.

\bibitem{banbas89}
R.~Bansal and T.~Ba\c{s}ar, ``Simultaneous design of measurement and control
  strategies in stochastic systems with feedback,'' \emph{Automatica}, vol.~45,
  pp. 679--694, September 1989.

\bibitem{MatveevSavkin}
A.~S. Matveev and A.~V. Savkin, \emph{Estimation and Control over Communication
  Networks}.\hskip 1em plus 0.5em minus 0.4em\relax Boston: Birkh\"auser, 2008.

\bibitem{SahaiParts}
A.~Sahai and S.~Mitter, ``The necessity and sufficiency of anytime capacity for
  stabilization of a linear system over a noisy communication link part {I}:
  {S}calar systems,'' \emph{IEEE Transactions on Information Theory}, vol.~52,
  no.~8, pp. 3369--3395, 2006.

\bibitem{minero2012stabilization}
P.~Minero, L.~Coviello, and M.~Franceschetti, ``Stabilization over markov
  feedback channels: the general case,'' \emph{IEEE Transactions on Automatic
  Control}, vol.~58, no.~2, pp. 349--362, 2012.

\bibitem{kostina2018exact}
V.~Kostina, Y.~Peres, G.~Ranade, and M.~Sellke, ``Exact minimum number of bits
  to stabilize a linear system,'' in \emph{2018 IEEE Conference on Decision and
  Control (CDC)}.\hskip 1em plus 0.5em minus 0.4em\relax IEEE, 2018, pp.
  453--458.

\bibitem{Martins2008}
N.~C. Martins and M.~A. Dahleh, ``Feedback control in the presence of noisy
  channels: '{B}ode-like' fundamental limitations of performance,'' \emph{IEEE
  Transactions on Automatic Control}, vol.~53, pp. 1604--1615, August 2008.

\bibitem{NairFagnaniSurvey}
G.~N. Nair, F.~Fagnani, S.~Zampieri, and J.~R. Evans, ``Feedback control under
  data constraints: an overview,'' \emph{Proceedings of the IEEE}, pp.
  108--137, 2007.

\bibitem{YukTAC2010}
S.~Y\"uksel, ``Stochastic stabilization of noisy linear systems with fixed-rate
  limited feedback,'' \emph{IEEE Transactions on Automatic Control}, vol.~55,
  pp. 2847--2853, December 2010.

\bibitem{YukMeynTAC2010}
S.~Y\"uksel and S.~P. Meyn, ``Random-time, state-dependent stochastic drift for
  {M}arkov chains and application to stochastic stabilization over erasure
  channels,'' \emph{IEEE Transactions on Automatic Control}, vol.~58, pp.
  47--59, January 2013.

\bibitem{johnston2013stochastic}
A.~P. Johnston and S.~Y{\"u}ksel, ``Stochastic stabilization of partially
  observed and multi-sensor systems driven by unbounded noise under fixed-rate
  information constraints,'' \emph{IEEE Transactions on Automatic Control},
  vol.~59, no.~3, pp. 792--798, 2013.

\bibitem{YukselAMSITArxiv}
S.~Y\"uksel, ``Characterization of information channels for asymptotic mean
  stationarity and stochastic stability of non-stationary/unstable linear
  systems,'' \emph{IEEE Transactions on Information Theory}, vol.~58, pp.
  6332--6354, October 2012.

\bibitem{YukselBasarBook}
S.~Y\"uksel and T.~Ba\c{s}ar, \emph{Stochastic Networked Control Systems:
  Stabilization and Optimization under Information Constraints}.\hskip 1em plus
  0.5em minus 0.4em\relax New York: Springer, 2013.

\bibitem{liberzon2005stabilization}
D.~Liberzon and J.~P. Hespanha, ``Stabilization of nonlinear systems with
  limited information feedback,'' \emph{IEEE Transactions on Automatic
  Control}, vol.~50, no.~6, pp. 910--915, 2005.

\bibitem{de2005n}
C.~De~Persis, ``n-bit stabilization of n-dimensional nonlinear systems in
  feedforward form,'' \emph{IEEE Transactions on Automatic Control}, vol.~50,
  no.~3, pp. 299--311, 2005.

\bibitem{nair2004topological}
G.~N. Nair, R.~J. Evans, I.~M. Mareels, and W.~Moran, ``Topological feedback
  entropy and nonlinear stabilization,'' \emph{IEEE Transactions on Automatic
  Control}, vol.~49, no.~9, pp. 1585--1597, 2004.

\bibitem{adler1965topological}
R.~L. Adler, A.~G. Konheim, and M.~H. McAndrew, ``Topological entropy,''
  \emph{Transactions of the American Mathematical Society}, vol. 114, no.~2,
  pp. 309--319, 1965.

\bibitem{colonius2009invariance}
F.~Colonius and C.~Kawan, ``Invariance entropy for control systems,''
  \emph{SIAM Journal on Control and Optimization}, vol.~48, no.~3, pp.
  1701--1721, 2009.

\bibitem{kawan2013invariance}
C.~Kawan, ``Invariance entropy for deterministic control systems,''
  \emph{Lecture notes in mathematics}, vol. 2089, 2013.

\bibitem{colonius2013note}
F.~Colonius, C.~Kawan, and G.~Nair, ``A note on topological feedback entropy
  and invariance entropy,'' \emph{Systems \& Control Letters}, vol.~62, no.~5,
  pp. 377--381, 2013.

\bibitem{colonius2018metric}
F.~Colonius, ``Metric invariance entropy and conditionally invariant
  measures,'' \emph{Ergodic Theory and Dynamical Systems}, vol.~38, no.~3, pp.
  921--939, 2018.

\bibitem{nair2007feedback}
G.~N. Nair, F.~Fagnani, S.~Zampieri, and R.~J. Evans, ``Feedback control under
  data rate constraints: An overview,'' \emph{Proceedings of the IEEE},
  vol.~95, no.~1, pp. 108--137, 2007.

\bibitem{matveev2009estimation}
A.~S. Matveev and A.~V. Savkin, \emph{Estimation and control over communication
  networks}.\hskip 1em plus 0.5em minus 0.4em\relax Springer Science \&
  Business Media, 2009.

\bibitem{kawan2019invariance}
C.~Kawan and S.~Y\"uksel, ``Invariance properties of nonlinear stochastic
  dynamical systems under information constraints,'' \emph{IEEE Transactions on
  Automatic Control, to appear (arXiv: 1901.02825)}, 2020.

\end{thebibliography}

\addtolength{\textheight}{-12cm}

\end{document}